\documentclass{article} % Use the article class for a standard document

\usepackage[dvipsnames,svgnames,table]{xcolor}

\usepackage[colorinlistoftodos,textsize=footnotesize,backgroundcolor=yellow!40,bordercolor=red,linecolor=red]{todonotes}

\usepackage{titling}   % For customizing title, author, and date formatting
\usepackage{titlesec}  % For modifying section titles

\usepackage{mathtools} % Enhances amsmath with additional features and fixes
\usepackage[skins,theorems]{tcolorbox} % For creating colored boxes (e.g., for cases, examples)
\usepackage{amsmath}   % Core package for advanced math typesetting
\usepackage{amsthm}    % Provides enhanced theorem-like environments
\usepackage{amssymb}   % Additional math symbols

\usepackage[a4paper, margin=3cm]{geometry}

\usepackage{tikz}

\usepackage{lineno}
\mathtoolsset{showonlyrefs}

\newtheorem{theorem}{Theorem} % Defines a theorem environment
[section] % Numbering of Theorem (can be modified to [section] if needed)
\newtheorem{lemma}[theorem]{Lemma} % Lemmas share the numbering with Theorems
\newtheorem{question}[theorem]{Question} % Lemmas share the numbering with Theorems
\newtheorem{proposition}[theorem]{Proposition} % Propositions share numbering with Theorems
\theoremstyle{definition} % Switch to a definition style (upright text)
\newtheorem{remark}[theorem]{Remark} % Remarks share numbering with Theorems

\newcommand{\R}{\mathbb{R}}    % Real numbers
\DeclareMathOperator*{\Div}{div}

\newcommand\norm[1]{\left\Arrowvert#1\right\Arrowvert} % Norm with double arrows
\newcommand\abs[1]{\left\vert#1\right\vert} % Absolute value

\usepackage[colorlinks=true, linkcolor=black, urlcolor=black, citecolor=DarkBlue, filecolor=black]{hyperref} % Enables clickable links in the document

\numberwithin{equation}{section}

\pretitle{\begin{center}\sffamily\huge}    % Before title: Center and set font size and style
\posttitle{\par\end{center}}               % End title formatting
\preauthor{\begin{center}\large\sffamily}   % Before author: Center and set style
\postauthor{\par\end{center}}              % End author formatting
\predate{\begin{center}\large\itshape\sffamily} % Before date: Center, italicize, and set style
\postdate{\par\end{center}}                % End date formatting

\usepackage{tocloft} % Allows customization of the Table of Contents
\title{A Discontinuous Solution of the Critical n-Laplace System with Antisymmetric Potential}          % Change to your document's title
\author{Dominik Schlagenhauf \thanks{Department of Mathematics, ETH Zentrum,
CH-8092 Z\"urich, Switzerland.}} % Change to your name or the author's name
\date{\today}                % Automatically inserts today's date

\begin{document}

\maketitle % Prints the title, author, and date as defined above
%\vspace{50mm} % Uncomment and adjust if additional space is needed after the title

\begin{abstract}
%%%%%%%%%%%% CHANGE ABSTRACT HERE %%%%%%%%%%%%
%%%%%%%%%%%% ↓ ↓ ↓ ↓ ↓ ↓ ↓ ↓ ↓ ↓  %%%%%%%%%%%%
%%%%%%%%%%%% ↓ ↓ ↓ ↓ ↓ ↓ ↓ ↓ ↓ ↓  %%%%%%%%%%%%
Let \(n>2\).  We construct a map
\(U\in W^{1,n}(B^n,\R^{n+2})\) that is discontinuous at the origin and smooth on the punctured ball $B^n \setminus \{0\}$, together with an antisymmetric potential
\[
 \Omega\in L^n\bigl(B^n,so(n+2)\otimes\R^n\bigr)
\]
such that
\[
 -\Div\bigl(|\nabla U|^{n-2}\nabla U\bigr)
 =\Omega\cdot |\nabla U|^{n-2}\nabla U
 \qquad\text{in }\mathcal D'(B^n).
\]
This gives a negative answer to a regularity question posed by Rivi\`ere in
\cite[Eq.~(3.23)]{RiviereSurvey}, and later reformulated in the open Problem 2.5 in the survey paper
\cite{SchikorraStrzelecki}.
Our potential admits the Lorentz-space regularity $\Omega \in \cap_{q>2}L^{(n,q)} \setminus L^{(n,2)}$.
In addition for given $1<p<\infty$ we can enforce $\nabla U \in L^{(n,p)}$ but $\nabla U \notin L^{(n,1)}$.
The construction does not give a counterexample to regularity for weakly \(n\)-harmonic maps or for
higher-dimensional \(H\)-systems. \\
The example was generated by ChatGPT 5.6 Sol. The work itself was written by the author and thoroughly reviewed to ensure its correctness. 
\end{abstract}
%\vspace{15mm} % Uncomment and adjust if additional vertical space is desired after the abstract

%%%%%%%%%%%% Generate the Table of Contents %%%%%%%%%%%%
\begin{comment}
\tableofcontents
\vspace{10mm} % Adds extra vertical space after the TOC
\end{comment}
%\newpage  % Changes to a new page before starting the document

%%%%%%%%%%%% BEGIN MAIN DOCUMENT HERE %%%%%%%%%%%%
%%%%%%%%%%%% ↓ ↓ ↓ ↓ ↓ ↓ ↓ ↓ ↓ ↓ ↓ ↓  %%%%%%%%%%%%
%%%%%%%%%%%% ↓ ↓ ↓ ↓ ↓ ↓ ↓ ↓ ↓ ↓ ↓ ↓  %%%%%%%%%%%%

\section{Introduction}
%%%%%%%%%%%%%%%%%% INTODUCTION %%%%%%%%%%%%%%%%%%%
%%%%%%%%%%%% ↓ ↓ ↓ ↓ ↓ ↓ ↓ ↓ ↓ ↓ ↓ ↓  %%%%%%%%%%%%
%%%%%%%%%%%% ↓ ↓ ↓ ↓ ↓ ↓ ↓ ↓ ↓ ↓ ↓ ↓  %%%%%%%%%%%%
%%%%%%%%%%%% ↓ ↓ ↓ ↓ ↓ ↓ ↓ ↓ ↓ ↓ ↓ ↓  %%%%%%%%%%%%
%%%%%%%%%%%% ↓ ↓ ↓ ↓ ↓ ↓ ↓ ↓ ↓ ↓ ↓ ↓  %%%%%%%%%%%%

Harmonic maps on surfaces have a central role in the field of geometric analysis as they are the critical points of the Dirichlet energy.
In the groundbreaking work \cite{Riviere2007} Rivi\`ere proved that these are continuous by rewriting the harmonic map equation and showing that solutions $u\in W^{1,2}(D^2,\R^m)$ of
\begin{equation}
-\Delta u = \Omega \cdot \nabla u
\end{equation}
are continuous,
where $\Omega \in L^2(D^2, so(m)\otimes \R^2)$ is an antisymmetric potential $1$-form and $D^2$ the unit disk.
A natural higher-dimensional analogue is obtained by replacing the
Laplacian by the conformally invariant \(n\)-Laplacian.
Let \(B^n\subset\R^n\) be the unit ball, \(n>2\), and consider
\begin{equation}\label{eq:systemwithn}
 -\Div\bigl(|\nabla u|^{n-2}\nabla u\bigr)
 =\Omega\cdot |\nabla u|^{n-2}\nabla u,
 \qquad
 u\in W^{1,n}(B^n,\R^m),
\end{equation}
where
\begin{equation}
 \Omega\in L^n(B^n,so(m)\otimes\R^n).
\end{equation}
Thus, \(\Omega=(\Omega_\alpha)_{\alpha=1}^n\) is an \(L^n\) matrix-valued
vector field, with each \(\Omega_\alpha(x)\) a antisymmetric
\(m\times m\) matrix.
Here and below
\[
 \bigl(\Omega\cdot |\nabla u|^{n-2}\nabla u\bigr)^i
 =
 \sum_{\alpha=1}^n\sum_{j=1}^N
 \Omega^{ij}_\alpha |\nabla u|^{n-2}\partial_\alpha u^j.
\]
In \cite[Eq.~(3.23)]{RiviereSurvey}, Rivi\`ere raised the following
regularity question.

\begin{question}\label{q:riviere}
Suppose that $u\in W^{1,n}$ is a weak solution to \eqref{eq:systemwithn} with $\Omega \in L^n$.
Is $u$ continuous or not?
\end{question}

We would like to emphasise that the regularity assumed on $u$ and $\Omega$ in Question \ref{q:riviere} is the natural regularity assumption for the distributional meaning of the PDE \eqref{eq:systemwithn}.

The same question was subsequently singled out by Schikorra and Strzelecki
as an open problem in their survey; see
\cite[Problem~2.5]{SchikorraStrzelecki}. They emphasized its relation to
two central regularity questions in higher-dimensional conformal geometry:
weak \(n\)-harmonic maps and \(H\)-systems.
Question \ref{q:riviere} was also posed in \cite{MartinoSchikorra}.

An instructive precursor is the counterexample of Firoozye \cite{Firoozye} for the scalar $n$-Laplace equation. He showed that, in dimensions $n>2$, the condition
\begin{equation}
\Delta_n u\in\mathcal H^1_{\mathrm{loc}}
\end{equation}
does not imply continuity. In particular, for
\begin{equation}
u(x)=\sin\!\left(\log^\alpha\frac1{|x|}\right),
\qquad 0<\alpha<1-\frac2n,
\end{equation}
one has \(\Delta_nu\in\mathcal H^1_{\mathrm{loc}}\), while \(u\) is bounded and discontinuous at the origin. The present problem is more structured: in \eqref{eq:systemwithn} the $n$-laplacian is required to componentwise factor as
\begin{equation}
\Omega\cdot|\nabla u|^{n-2}\nabla u.
\end{equation}
Thus Firoozye's example shows that critical Hardy-space control of the \(n\)-tension field alone does not yield continuity, whereas Rivi\`ere's question asks whether the additional antisymmetric first-order structure is sufficient.

The question is subtle because the two-dimensional mechanism does not
survive at the same level of generality in higher dimensions.  A number of
positive results are known under additional structure.  For instance,
\(n\)-harmonic maps into spheres and, more generally, certain homogeneous
targets enjoy full regularity; see, among others,
\cite{Strzelecki1994,MouYang}.

In the work of Martino and Schikorra \cite{MartinoSchikorra} they assume, in addition to
antisymmetry,
\begin{equation}
\label{eq:OmegainLn2}
 \Omega\in L^{(n,2)}
\end{equation}
together with the zero-order curl condition
\begin{equation}
\label{eq:RieszPotential}
 \max_{i,j}\max_{\alpha,\gamma}
 \bigl\|
 \mathcal R_\alpha\Omega^{ij}_\gamma
 -
 \mathcal R_\gamma\Omega^{ij}_\alpha
 \bigr\|_{L^{(n,1)}}<\infty,
\end{equation}
where \(\mathcal R_\alpha=\partial_\alpha(-\Delta)^{-1/2}\) denotes the
Riesz transform.  
Under these assumptions they prove continuity of $u$.
They also show
that the geometric structure of the \(n\)-harmonic-map equation allows one
to verify the required transformed-potential estimates from the single
additional hypothesis
\begin{equation}
\label{eq:nab u in Ln2}
 \nabla u\in L^{(n,2)},
\end{equation}
and hence obtain continuity of weakly \(n\)-harmonic maps which satisfy \eqref{eq:nab u in Ln2}.
However, the general regularity question of $n$-harmonic maps in $W^{1,n}$ remains open.
Our discontinuous solution can be chosen to satisfy
\(\nabla U\in L^{(n,2)}\), but its potential $\Omega$ never belongs to
\(L^{(n,2)}\).  Thus the \(L^{(n,2)}\)-integrability of the gradient alone
has no regularizing effect for the unrestricted system
\eqref{eq:systemwithn}.

\medskip
The purpose of the present paper is to show that, at this level of
generality, the answer to Rivi\`ere's Question \ref{q:riviere} is negative.  The
counterexample is explicit and has only one singular point.

\begin{theorem}
\label{thm:main thm}
Let $n>2$ and let $1<p\le \infty$.
Suppose $\max(2^{-1},p^{-1})<\beta<1$.
Then there exists some $T_0>0$ such that for any $T\ge T_0$ the map
\begin{equation}\label{eq:U-intro}
 U(x)=
 \left(
 \cos\bigl(t^{1-\beta}\bigr),
 \sin\bigl(t^{1-\beta}\bigr),
 t^{-(\beta+1/2)}\frac{x}{|x|}
 \right),
 \qquad
 t=T+\log\frac1{|x|},
\end{equation}
defined for \(x\in B^n\setminus\{0\}\) satisfies
\begin{equation}
\nabla U \in L^{(n,p)}(B^n), \qquad U\in L^{\infty}(B^n)
\end{equation}
\underline{but is discontinuous at the origin}.
Furthermore, there exists a potential 
\begin{equation}
\Omega\in \bigcap_{q>2} L^{(n,q)}\bigl(B^n,so(n+2)\otimes\R^n\bigr)
\end{equation}
such that
\begin{equation}\label{eq:main-equation}
 -\Div\bigl(|\nabla U|^{n-2}\nabla U\bigr)
 =\Omega\cdot|\nabla U|^{n-2}\nabla U
 \qquad\text{in }\mathcal D'(B^n).
\end{equation}
\end{theorem}

\begin{remark}
The example constructed in Theorem \ref{thm:main thm} is not a counterexample to the regularity of weakly $n$-harmonic maps and small modifications of it don't lead to one. 
As explained in \cite{SchikorraStrzelecki} a counterexample of this kind would be required to be singular on a perfect set, rather than on an isolated point. 
\end{remark}

\begin{remark}
While preparing this manuscript for submission to arXiv, the author became aware of the contemporaneous work \cite{MartinoSchikorra2026}, which independently gives a counterexample to Question \ref{q:riviere}.
In fact, both papers treat the same family of examples constructed in Theorem \ref{thm:main thm}.
\end{remark}

\noindent
\textbf{Discussion} 
The first two components in \eqref{eq:U-intro} keep winding around $S^1$ as $|x|\to0$ and are therefore responsible for the discontinuity at the origin.
These two first component generate a non-zero component of the $n$-tension field that is parallel to $\partial_r U$. 
But we note that whenever $A$ is an antisymmetric matrix one hat $A\partial_r U \perp \partial_r U$ and therefore the $n$-tension field can not be generated from $\partial_r U$ alone. 
This is precisely the role that the remaining $n$-components have in \eqref{eq:U-intro}.
These generate $n-1$ additional directions in the image of $\nabla U$.
Their size is chosen so that producing
the missing parallel component costs
\[
 \frac1{|x|\,t^{1/2}}
\]
to the potential. 
This term is \(n\)-integrable exactly when \(n>2\), and it is also
responsible for the universal Lorentz obstruction
\(\Omega\notin L^{(n,2)}\), independently of \(\beta\).
The parameter $\beta$ controls the winding speed with
\begin{equation}
|\nabla U| \simeq \frac{1}{|x| t^{\beta}},
\end{equation}
which reveals the regularity $\nabla U\in L^{n,p}$ whenever $np>1$.
The underlaying idea in the construction of the potential $\Omega$ is to define it via a wedge product of the form $(X\wedge Y)/|Y|^2$, where $X$ is related to the $n$-tension field of $U$ and $Y$ to the gradient $\nabla U$. One then exploits the fact that whenever $X\perp Y$ the identity 
\begin{equation}
\frac{(X\wedge Y)}{|Y|^2} Y=X
\end{equation}
holds, permitting to gain back the $n$-tension field. 

\medskip
\noindent
\textbf{AI Usage}
The example was provided by ChatGPT 5.6 Sol while the author was exploring possible counterexamples to the regularity problem of weakly $n$-harmonic maps on August 5, 2026.
The author identified it as a solution to the more general problem with the antisymmetric potential as in \eqref{eq:systemwithn}.
Furthermore, the author simplified the paramters and notations for better readability.
The proof has been checked by the author and is correct.
The work was written by the author, however code snippets may occasionally come from LLMs including ChatGPT 5.6 Sol or Gemini 3.6 Thinking.

\section{Construction}

\subsection{The Map $U$ and its $n$-tension Field}

We will choose $T_0>0$ later down the line.
Suppose at the moment $T\ge T_0$.
To $x\in B^n\setminus \{0\}$ introduce the notation
\begin{equation}
r=|x|, \qquad
t=T + \log\left(\frac{1}{r}\right), \qquad
\omega= \frac{x}{|x|}.
\end{equation}
For a given function $V$ derviatves transform as
\begin{equation}
\label{eq:der r to t trans}
\partial_r V = -\frac{1}{r}\ \partial_t V.
\end{equation}
In what follows we introduce some additional notation which will be useful to simplify computations.
To any $\xi \in \R^n$ set
\begin{equation}
\widehat \xi =(0,\xi) \in \R^{n+2}.
\end{equation}
We introduce the angle
\begin{equation}
\phi(t)= t^{1-\beta}
\end{equation}
and the angular decay of the function
\begin{equation}
a(t) = t^{-(\beta+1/2)}.
\end{equation}
We also define the angular speed as
\begin{equation}
v(t) = (1-\beta) t^{-\beta} =\phi_t(t)
\end{equation}
We define the curves
\begin{equation}
q(t) = (\cos \phi(t), \sin\phi(t),0) \in \R^{n+2}
\end{equation}
and also
\begin{equation}
\tau(t) = (-\sin \phi(t), \cos\phi(t),0) \in \R^{n+2}
\end{equation}
such that
\begin{equation}
q_t(t)= v(t) \tau(t).
\end{equation}
We now have with $U$ as in Theorem \ref{thm:main thm}
\begin{equation}
U=q +a \widehat \omega
\end{equation}
and clearly $U\in L^{\infty}(B^n)\cap C^{\infty}(B^n\setminus \{0\})$.
We now show that $U$ is discontinuous at the origin:
Fix \(\omega_0\in S^{n-1}\), and choose
\[
 t_k=(2\pi k)^{1/(1-\beta)},
 \qquad
 \widetilde t_k=((2k+1)\pi)^{1/(1-\beta)}.
\]
The points
\[
 x_k=e^{T-t_k}\omega_0,
 \qquad
 \widetilde x_k=e^{T-\widetilde t_k}\omega_0
\]
converge to zero, whereas \(a(t)\to0\) (for $t\to\infty$) and therefore
\[
 U(x_k)\to(1,0,0),
 \qquad
 U(\widetilde x_k)\to(-1,0,0),
\]
as $k\to\infty$.
Hence, \(U\) has no limit at the origin.
\par 

We will now introduce polar coordinates:
Let $e_2,\dots,e_{n} \in T_{e_1} S^{n-1}$ be an orthonormal frame where $e_1=\omega(x)=x/|x|\in S^{n-1}$.
We have the explicit derivatives
\begin{equation}
\label{eq:derivatives of U}
\partial_{e_1} U = \partial_{r} U = - \frac{1}{r} (v\tau + a_t \widehat\omega)\in \R^{n+2}, \qquad
\partial_{e_j} U = \frac{a}{r} \widehat e_j \in \R^{n+2},
\qquad j=2,\dots,n.
\end{equation}
It immediatley follows that
\begin{equation}
|\nabla U|^2 = r^{-2} (v^2 + a_t^2 + (n-1)a^2).
\end{equation}
This gives for $T_0$ large enough
\begin{equation}
|\nabla U|\lesssim r^{-1}t^{-\beta} = \frac{1}{|x| \log(1/|x|)^{\beta}}
\end{equation}
and therefore by the critical Lorentz-space (see \cite{Grafakos2014}, chapter 1.4) estimate
\begin{equation}
\nabla U \in L^{n,p}(B^n) \iff p\beta>1.
\end{equation}
In particular $\nabla U \in L^{n,p}(B^n)$ for the choice of $p$ and $\beta$ as in Theorem \ref{thm:main thm}.
Letting
\begin{equation}
K(t) = (v^2+a_t^2+(n-1)a^2)^{(n-2)/2}
\end{equation}
one has
\begin{equation}
\label{eq:decomp of nabla U n-2}
|\nabla U|^{n-2}= |\nabla_{\mathcal E} U|^{n-2} = r^{-(n-2)} K(t),
\end{equation}
where clearly $K$ depends only on $t$.
We now compute the $n$-tension field:

\begin{lemma}
\label{lemma:n tension field}
For $U$ as in \eqref{eq:U-intro} one has to $x\in B^n$
\begin{equation}
\label{eq:n-tension field}
-\Div(|\nabla U|^{n-2} \nabla U) = r^{-n} Kg,
\end{equation}
where we introduced the normalized $n$-tension field
\begin{equation}
\label{eq:norm ntension field g}
g = v^2 q - \frac{\partial_t(Kv)}{K} \tau + \left( (n-1)a - \frac{\partial_t(Ka_t)}{K} \right) \widehat \omega.
\end{equation}
\end{lemma}

\begin{proof}
Using the standard divergence formula in polar coordinates, applied componentwise to $U$, we have
\begin{equation}
\Div\bigl(|\nabla U|^{n-2}\nabla U\bigr)
=
\frac{1}{r^{n-1}}
\partial_r\!\left(
r^{n-1}|\nabla U|^{n-2}\partial_r U
\right)
+\frac{1}{r^2}
\Div_{S^{n-1}}\!\left(
|\nabla U|^{n-2}
\nabla_{S^{n-1}}U
\right).
\end{equation}
We write with \eqref{eq:der r to t trans}, \eqref{eq:derivatives of U} and \eqref{eq:decomp of nabla U n-2}
\begin{equation}
\begin{aligned}
\frac{1}{r^{n-1}}
\partial_r\!\left(
r^{n-1}|\nabla U|^{n-2}\partial_r U
\right)
&= r^{1-n}\partial_r\!\left(
r K \partial_r U
\right) \\
&= r^{-n} \partial_t(K(v\tau+a_t\widehat\omega)) \\
&= r^{-n} \Big( \partial_t(Kv)\tau - Kv^2q
+ \partial_t(K a_t) \widehat\omega \Big) \\
&= r^{-n} K \Bigg(\frac{\partial_t(Kv)}{K} \tau - v^2q
+ \frac{\partial_t(K a_t)}{K} \widehat\omega \Bigg) 
\end{aligned}
\end{equation}
Since, $K$, $q$ and $a$ depend only on $t$ we find
\begin{equation}
\begin{aligned}
\frac{1}{r^2}
\Div_{ S^{n-1}}\!\left(
|\nabla U|^{n-2}
\nabla_{\mathbb S^{n-1}}U
\right)
&=\frac{1}{r^2}
\Div_{S^{n-1}}\!\left(
r^{-(n-2)} K\
\nabla_{ S^{n-1}}U
\right) \\
&=r^{-n} K \Delta_{S^{n-1}} (q+a\widehat\omega) \\
&=r^{-n} K a\Delta_{S^{n-1}} (\widehat\omega) \\
&=-(n-1)r^{-n} K a\ \widehat\omega \\
&=r^{-n}K \big(-(n-1) a\ \widehat\omega \big)
\end{aligned}
\end{equation}
This proves \eqref{eq:n-tension field} and the lemma.
\end{proof}

\subsection{The Antisymmetric Potential $\Omega$ and the PDE}

We start by decomposing $g$
\begin{equation}\label{eq:g-decomp}
 g^{\parallel}
 =
 \frac{\langle g,\partial_t U \rangle}{|\partial_t U|^2}\partial_t U,
 \qquad
 g^{\perp}=g-g^{\parallel}.
\end{equation}

\begin{lemma}\label{lem:g-estimates}
There exists some $T_0>0$ such that for any $t\ge T\ge T_0$
\begin{equation}
\label{eq: g-per} 
 \frac{|g^{\perp}|}{|\partial_tU|}
 \leq
 C\ t^{-1/2},
\end{equation}
as well as
\begin{equation}
\label{eq: g-par}
 \frac{|g^{\parallel}|}{a}
 \leq
 C\ t^{-1/2},
\end{equation}
where $C>0$ may depend at most on $n$, $\beta$ and $T_0$.
In particular, the implicit constants in \eqref{eq: g-per} and \eqref{eq: g-par} are positive
and uniform for \(t\geq T\).
\end{lemma}

\begin{proof}
We start by calculating the ratios
\begin{equation}\label{eq:ratios}
 \frac{a}{v}=\frac{1}{1-\beta}t^{-1/2},
 \qquad
 \frac{a_t}{v}
 =-\frac{\beta+1/2}{1-\beta}t^{-3/2}.
\end{equation}
By definition we have
\begin{equation}\label{eq:K-exact}
 K
 =
 v^{n-2}
 \left(
q(t)
 \right)^{(n-2)/2},
\end{equation}
where
\begin{equation}
q(t)=
 1+\frac{n-1}{(1-\beta)^2} t^{-1}
 +\frac{(\beta+1/2)^2}{(1-\beta)^2} t^{-3}.
\end{equation}
This implies that for $t\ge T_0(n,\beta)>0$ large enough we can enforce $q\in (1/2,3/2)$ say.
One has
\begin{equation}
\begin{aligned}
\label{eq:Kt-devide-by-K}
\frac{K_t}{K}
&= \partial_t \log K =(n-2)\partial_t (\log v) + \frac{n-2}{2} \partial_t (\log q) \\
&=(n-2)\frac{v_t}{v} + \frac{n-2}{2} \frac{q_t}{q} \\
&=-(n-2)\beta\ t^{-1} + \frac{n-2}{2} \frac{q_t}{q},
\end{aligned}
\end{equation}
where here $q_t(t)= \mathcal O (t^{-2})$ as $t\to\infty$.
Moreover,
\begin{equation}
\label{eq:partialtUequivtov}
|\partial_tU|= v \left(1 + \frac{a_t^2}{v^2} \right)^{1/2}
\end{equation}
and hence with \eqref{eq:ratios} for $t\ge T_0(n,\beta)$ we have
\begin{equation}
\label{eq:growth of Ut}
|\partial_tU|\simeq v.
\end{equation}

We compute with \eqref{eq:der r to t trans}, \eqref{eq:derivatives of U} and \eqref{eq:norm ntension field g}
\begin{equation}
\begin{aligned}
\frac{|g^{\parallel}|}{a}
&= \frac{1}{a|\partial_tU|} |\langle g,\partial_t U \rangle|
= \frac{1}{a|\partial_tU|} |\langle g, v\tau + a_t \widehat\omega\rangle| \\
&\le \Bigg|\frac{\partial_t(Kv)}{Ka}\Bigg| \frac{v}{|\partial_tU|} + \left| (n-1) - \frac{\partial_t(Ka_t)}{aK} \right| \frac{|a_t|}{|\partial_tU|} \\
&\le  \left|\underbrace{\frac{K_t}{K}}_{\simeq t^{-1}+\mathcal O(t^{-2})} \underbrace{\frac{v}{a}}_{\simeq t^{1/2}} + \underbrace{\frac{v_t}{a}}_{\simeq t^{-1/2}} \right| \underbrace{\frac{v}{|\partial_tU|}}_{\simeq 1} \\
&\qquad\qquad + \left| (n-1) - \left( \underbrace{\frac{K_t}{K}}_{\simeq t^{-1}} \underbrace{\frac{a_t}{a}}_{\simeq t^{-1}} + \underbrace{\frac{a_{tt}}{a}}_{\simeq t^{-2}} \right) \right| \underbrace{\frac{|a_t|}{|\partial_tU|}}_{\simeq t^{-3/2}},
\end{aligned}
\end{equation}
where in the last line we used \eqref{eq:ratios}, \eqref{eq:K-exact}, \eqref{eq:Kt-devide-by-K}, \eqref{eq:growth of Ut}.
This shows \eqref{eq: g-par}.
Now we may use the rough estimate
\begin{equation}
\begin{aligned}
\frac{|g^\perp|}{|\partial_tU|}
&\le \frac{|g|}{|\partial_tU|}
\le \frac{1}{|\partial_tU|} \left( v^2 + \left|\frac{K_t}{K}v\right| +|v_t| + (n-1)a + \left|\frac{K_t}{K}a_t\right| + |a_{tt}| \right) \\
& \le \left( \underbrace{\frac{v}{|\partial_tU|}}_{\simeq1} \underbrace{v}_{\simeq t^{-\beta}} + \underbrace{\left|\frac{K_t}{K}\right|}_{\simeq t^{-1}} \underbrace{\frac{v}{|\partial_tU|}}_{\simeq 1} + \underbrace{\frac{|v_t|}{|\partial_tU|}}_{\simeq t^{-1}} + (n-1) \underbrace{\frac{a}{|\partial_tU|}}_{\simeq t^{-1/2}}
+ \underbrace{\left|\frac{K_t}{K}\right|}_{\simeq t^{-1}} \underbrace{\frac{|a_t|}{|\partial_tU|}}_{\simeq t^{-3/2}} + \underbrace{\frac{|a_{tt}|}{|\partial_tU|}}_{\simeq t^{-5/2}}
\right)
\end{aligned}
\end{equation}
The bound \eqref{eq: g-per} follows.
\end{proof}

We now define the potential intrinsically, without using the tangent frame on the sphere.
Define the orthogonal projection of \(\R^n\) onto \(T_\omega S^{n-1}\) by
\begin{equation}
 P_\omega X=X-\langle X,\omega\rangle\omega.
\end{equation}
For \(X,Y\in\R^m\), define
\begin{equation}\label{eq:wedge}
 (X\wedge Y)Z
 =
 X\langle Y,Z\rangle-Y\langle X,Z\rangle.
\end{equation}
Then \(X\wedge Y\in so(m)\).  If \(X\perp Y\), then
\begin{equation}\label{eq:wedge-action}
 (X\wedge Y)Y=|Y|^2X.
\end{equation}
Whenever working in coordinates $x,y\in\R^m$ we may express the wedge product as
\begin{equation}
x \wedge y =x y^{tr} -y x^{tr},
\end{equation}
where the products are to be understood as matrix multiplication and $tr$ denotes the transpose of a matrix.
To $x\in B^n\setminus \{0\}$ and any vector $X\in\R^n$ set
\begin{equation}
 \Omega(x)[X]
 \coloneqq
 \frac{g^{\perp}\wedge \partial_{r}U}{r^2|\partial_{r}U|^2} \langle X,\omega\rangle
+ \frac{ g^{\parallel}\wedge\widehat{P_\omega X}}{(n-1)ar}
\end{equation}
Thus \(\Omega(x)[X]\in so(n+2)\), and \(\Omega\) is a smooth
\(so(n+2)\)-valued one-form on \(B^n\setminus\{0\}\).

\begin{proposition}\label{prop:pointwiseOmega}
The potential \(\Omega\) satisfies
\begin{equation}\label{eq:pointwiseOmega}
 -\Div\bigl(|\nabla U|^{n-2}\nabla U\bigr)
 =
 \Omega\cdot|\nabla U|^{n-2}\nabla U
\end{equation}
pointwise on \(B^n\setminus\{0\}\).  Moreover,
\begin{equation}\label{eq:Omega-est}
 |\Omega(x)|
\le \frac{C}{r} t^{-1/2}
\end{equation}
and thus $\Omega\in \bigcap_{q>2} L^{(n,q)}(B^n)$.
\end{proposition}

\begin{proof}
Notice that to compute the right-hand side of \eqref{eq:pointwiseOmega} in coordinates we may verify this with our orthonormal basis of $\R^n$ of choice. We will use the basis $\{e_1,\dots,e_n\}$.
For any $X\in \R^n$ one has
\begin{equation}
 \Omega(x)[X] =
 \frac{g^{\perp}\wedge \partial_{e_1}U}{r^2|\partial_{e_1}U|^2} \langle X,e_1\rangle \\
 + \frac{1}{n-1} \sum_{j=2}^{n} \frac{g^{\parallel}\wedge \partial_{e_j}U}{r^2|\partial_{e_j}U|^2} \langle X,e_j\rangle.
\end{equation}
and thus with \eqref{eq:derivatives of U} and \eqref{eq:der r to t trans}
\begin{equation}
\begin{aligned}
\label{eq:Omegae1j}
 \Omega(x)[e_1]
 &=
 -\frac1r
 \frac{g^{\perp}\wedge \partial_t U}{|\partial_tU|^2}, \\
 \Omega(x)[e_j]
 &=
 \frac{1}{(n-1)ar}
 g^{\parallel}\wedge\widehat{e_j},
 \qquad j=2,\dots,n.
\end{aligned}
\end{equation}
According to \eqref{eq:decomp of nabla U n-2}, \eqref{eq:derivatives of U}, \eqref{eq:der r to t trans}
\begin{equation}
\label{eq:comp of rhs}
\Omega\cdot|\nabla U|^{n-2}\nabla U
= r^{-(n-2)}K
 \left[
 \Omega[e_1]\left(-\frac1r \partial_tU\right)
 +\sum_{j=2}^{n}
 \Omega[e_j]\left(\frac ar\widehat{e_j}\right)
 \right].
\end{equation}
Since \(g^{\perp}\perp \partial_t U\), with \eqref{eq:wedge-action}
\begin{equation}
 \Omega[e_1]\left(-\frac1r\partial_tU\right)
=\frac1{r^2}g^{\perp}.
\end{equation}
Furthermore, \(g^{\parallel}\perp\widehat{e_j}\), so
\begin{equation}
 \Omega[e_j]\left(\frac ar\widehat{e_j}\right)
= \frac1{(n-1)r^2}g^{\parallel}.
\end{equation}
Going back to \eqref{eq:comp of rhs} we have found
\begin{equation}
\Omega\cdot|\nabla U|^{n-2}\nabla U
= r^{-n}Kg,
\end{equation}
which according to Lemma \ref{lemma:n tension field} yields \eqref{eq:pointwiseOmega}.
For the upper bound, \eqref{eq:Omegae1j} implies
\[
|\Omega|
\le \frac Cr
\left(
 \frac{|g^{\perp}|}{|U_t|}
 +\frac{|g^{\parallel}|}{a}
 \right),
\]
and Lemma~\ref{lem:g-estimates} gives
\begin{equation}
 |\Omega|
 \leq \frac Cr t^{-1/2} = \frac{C}{|x| \log(1/|x|)^{1/2}}
\end{equation}
and therefore by the critical Lorentz-space (see \cite{Grafakos2014}, chapter 1.4) estimate
\begin{equation}
\Omega\in \bigcap_{q>2} L^{(n,q)}\bigl(B^n,so(n+2)\otimes\R^n\bigr).
\end{equation}
\end{proof}

\begin{proposition}\label{prop:weak}
Identity \eqref{eq:pointwiseOmega} holds in $\mathcal D'(B^n)$.  Equivalently,
for every test map
\[
 \varphi\in C_c^\infty(B^n,\R^{n+2}),
\]
one has
\begin{equation}\label{eq:weak}
 \int_{B^n}
 |\nabla U|^{n-2}\nabla U\cdot\nabla\varphi\,dx
 =
 \int_{B^n}
 \bigl(\Omega\cdot|\nabla U|^{n-2}\nabla U\bigr)
 \cdot\varphi\,dx.
\end{equation}
\end{proposition}

\begin{proof}
This result follow immediately by the theory of Sobolev-capacity (see e.g. \cite{EvansGariepy2015}).
We give a simple proof here for completeness. 
Let $\varepsilon>0$, let $t_\varepsilon=t(\varepsilon)=T + \log\left(1/\varepsilon\right)$ and let $B_\varepsilon$ be the ball of radius $\varepsilon$ centered at the origin.
According to \eqref{eq:pointwiseOmega} we have
\begin{equation}
\label{eq:weakEpsilon}
 \int_{B^n\setminus B_\varepsilon}
 \bigl(\Omega\cdot|\nabla U|^{n-2}\nabla U\bigr)
 \cdot\varphi\,dx
=  \int_{B^n\setminus B_\varepsilon}
 |\nabla U|^{n-2}\nabla U\cdot\nabla\varphi\,dx
- \int_{\partial B_\varepsilon} |\nabla U|^{n-2}\nabla U\cdot\omega\ \varphi\ d\sigma
\end{equation}
We estimate with \eqref{eq:decomp of nabla U n-2}, \eqref{eq:K-exact}, \eqref{eq:partialtUequivtov}
\begin{equation}
\begin{aligned}
\abs{\int_{\partial B_\varepsilon} |\nabla U|^{n-2}\nabla U\cdot\omega\ \varphi\ d\sigma}
&\le C \varepsilon^{-n+2} K(t_\varepsilon)\ \varepsilon^{n-1} \norm{\partial_{r} U}_{L^\infty(\partial B_\varepsilon)} \norm{\varphi}_{L^\infty(B^n)} \\
&= C K(t_\varepsilon) \norm{\partial_{t} U}_{L^\infty(\partial B_\varepsilon)} \norm{\varphi}_{L^\infty(B^n)} \\
&\le C v(t_\varepsilon)^{n-1} \norm{\varphi}_{L^\infty(B^n)} \\
&\le C t_\varepsilon^{-(n-1)\beta} \norm{\varphi}_{L^\infty(B^n)} \to 0,
\end{aligned}
\end{equation}
as $\varepsilon\searrow0$.
This permits to pass to the limit $\varepsilon\searrow0$ in \eqref{eq:weakEpsilon} to show \eqref{eq:weak}. 
\end{proof}
This completes the proof of Theorem \ref{thm:main thm}.

%%%%%%%%%%%%%%%%%%%%%%%%%%%%%%%%%%%%%%%%%%%%%%%%%%%%%%%%%%%%%%%
% Bibliography Section
%%%%%%%%%%%%%%%%%%%%%%%%%%%%%%%%%%%%%%%%%%%%%%%%%%%%%%%%%%%%%%%

\end{document}